\documentclass[a4paper]{amsart}

\usepackage[foot]{amsaddr}
\usepackage{graphicx}
\usepackage{amsmath,amssymb,amsfonts,amsthm}
\usepackage{float}
\usepackage{hyperref}
\usepackage{fullpage}

\floatstyle{ruled}
\newfloat{algorithm}{htb}{alg}
\floatname{algorithm}{Algorithm}

\theoremstyle{plain}
\newtheorem{theorem}{Theorem}
\newtheorem{lemma}{Lemma}
\newtheorem{corollary}{Corollary}

\theoremstyle{definition}

\theoremstyle{remark}

\newcommand{\reals}{\mathbb{R}}

\DeclareMathOperator{\rk}{rk}

\DeclareMathOperator{\cl}{span}
\DeclareMathOperator*{\argmax}{arg\,max}

\DeclareMathOperator*{\pos}{pos}

\title{Incremental Network Design with Minimum Spanning Trees}

\author{Konrad Engel$^1$}
\address{$^1$Institut f\"ur Mathematik, Universit\"at Rostock, 18051 Rostock, Germany} 

\author{Thomas Kalinowski$^2$}
\address{$^2$School of mathematical and Physical Sciences, University of Newcastle, Callaghan, NSW 2308, Australia}

\author{Martin W.P. Savelsbergh$^3$}
\address{$^3$School of Industrial \& Systems Engineering, Georgia Institute of Technology, Atlanta, U.S.A.}

\begin{document}

\begin{abstract}
  Given an edge-weighted graph $G=(V,E)$ and a set $E_0\subset E$, the incremental network design
  problem with minimum spanning trees asks for a sequence of edges $e'_1,\ldots,e'_T\in E\setminus
  E_0$ minimizing $\sum_{t=1}^Tw(X_t)$ where $w(X_t)$ is the weight of a minimum spanning tree $X_t$
  for the subgraph $(V,E_0\cup\{e'_1,\ldots,e'_t\})$ and $T=\lvert E\setminus E_0\rvert$. We prove
  that this problem can be solved by a greedy algorithm.
\end{abstract}

\maketitle

\section{Introduction}\label{sec:intro}

Network planning involves two stages. First, the structure of the network needs to be
decided. Second, the construction of the network needs to be scheduled. The first stage, the network
design stage, has received considerable attention in the operations research literature (see the
survey papers~\cite{kerivin2005design,magnanti1984network} and the references therein). The second
stage, the network construction stage, has received less attention. However, because the
construction of a network often stretches over a long period of time, the sequence in which the
network is constructed is important as it defines when specific parts of the network become
operational. It may even be beneficial to construct temporary links, i.e., links that are not part
of the ultimate network, in order for parts of the network to become operational.

Recently, there has been increased interest in problems that integrate a scheduling component into
network design problems, for instance motivated by the restoration of infrastructure networks after
disruptions~\cite{averbakh2012emergency,averbakh2012flowtime,AverbakhPereira-2015-Networkconstructionproblems,cavdaroglu2013integrating,NurreCavdarogluMitchellSharkeyWallace-2012-Restoringinfrastructuresystems,SharkeyCavdarogluNguyenHolmanMitchellWallace-2015-Interdependentnetworkrestoration}. A
classification of such \emph{integrated network design and scheduling} problems, with respect to underlying
scheduling environments and network performance measures was presented
in~\cite{NurreSharkey-2014-Integratednetworkdesign} along with complexity results and heuristic
algorithms.

One of the most basic variants of the framework from~\cite{NurreSharkey-2014-Integratednetworkdesign}
was studied under the name \emph{incremental network design problem}
in~\cite{BaxterElgindyErnstKalinowskiSavelsbergh-2014-Incrementalnetworkdesign}, in order to gain
insights into the trade-offs between construction cost and operational benefit. More specifically, an incremental network design problem can be associated
with any network optimization problem $P$, e.g., finding a shortest path, finding a maximum flow,
etc. In the most basic version, in addition to the network optimization problem $P$, an instance is
given by a network $G=(V,E)$ with vertex set $V$ and edge set $E$ and an existing edge set
$E_0$. The edge set $E\setminus E_0$ is referred to as the potential edge set and its cardinality
$T=\lvert E\setminus E_0\rvert$ as the planning horizon. Let $\varphi_P(G)$ denote the value of an
optimal solution to network optimization problem $P$ on network $G$. We are seeking a sequence
$E_0\subset E_1\subset\cdots\subset E_T=E$ with $\lvert E_i\setminus E_{i-1}\rvert=1$ giving rise to
networks $G_0,\,G_1,\ldots,G_T=G$, such that $\sum_{t=1}^T\varphi_P(G_t)$ is minimum (assuming that
$P$ is a minimization problem). That is, in the basic version, a single edge can be built in each
period of the planning horizon and we are seeking to minimize the operational costs over the
planning horizon.

This setting should be considered as a first purely mathematical step towards real-world
applications. In more elaborate versions, a construction cost may be associated with building a
potential edge and a budget may be available in each period, and the objective is to minimize the
operational costs over the planning horizon subject to the constraint that the construction cost of
the set of potential edges built in a period does not exceed the budget in that period.

Two natural heuristics for incremental network design problems, {\bf quickest-improvement} and {\bf
  quickest-to-ultimate}, are also of interest.  Quickest-improvement always seeks to improve the
value of the solution to the network optimization as quickly as possible, i.e., by adding as few
potential edges to the network as possible. A description of quickest-improvement can be found in
Algorithm~\ref{alg:qi}.
\begin{algorithm}[htb]
\caption{{\bf quickest-improvement}}\label{alg:qi}
\begin{tabbing}
....\=....\=....\=....\=................... \kill \\
$i \leftarrow 0$ ; $E' \leftarrow E_0$ \\
\textbf{while} $\varphi_P(G_{E'}) > \varphi_P(G_{E}) $ \textbf{do}  \\
\> $ k \leftarrow \min\left\{\lvert E'' \rvert\quad :\quad E'' \subseteq E \setminus E',\  \varphi_P(G_{E'})-\varphi_P(G_{E' \cup E''}) > 0 \right\}  $ \\
\> $ i \leftarrow i+1$; $E_i \leftarrow \argmax\left\{ \varphi_P(G_{E'}) - \varphi_P(G_{E' \cup E''})\quad :\quad  E'' \subseteq E \setminus E',\ \lvert E'' \rvert = k\right\} $ \\
\> $ E' \leftarrow E' \cup E_i$ \\
return ($E_1, \ldots, E_i, E \setminus \bigcup_{j=0}^i E_j$)
\end{tabbing}
\end{algorithm}

Quickest-to-ultimate first finds an optimal solution to the network optimization on the complete
network, referred to as an {\it ultimate} solution, and then always seeks to improve the value of
the solution to the network optimization as quickly as possible, but choosing only potential edges
that are part of the ultimate solution. A description of quickest-to-ultimate can be found in
Algorithm~\ref{alg:qtu}.
\begin{algorithm}[htb]
\caption{{\bf quickest-to-ultimate}}\label{alg:qtu}
\begin{tabbing}
....\=....\=....\=....\=................... \kill \\
Let $\overline{E}\subseteq E$ be a set of minimum cardinality such that $\varphi_P\left(\overline{E}\right)=\varphi_P(E)$\\
$i \leftarrow 0$ ; $E' \leftarrow E_0$ \\
\textbf{while} $\varphi_P(G_{E'}) > \varphi_P(G_{E}) $ \textbf{do}  \\
\> $ k \leftarrow \min\left\{ \lvert E'' \rvert\quad :\quad E'' \in \overline{E} \setminus E',\ \varphi_P(G_{E'}) - \varphi_P(G_{E' \cup E''}) > 0\right\} $ \\
\> $ i \leftarrow i+1$; $E_i \leftarrow \argmax\left\{ \varphi_P(G_{E'}) - \varphi_P(G_{E' \cup E''})\quad :\quad E'' \subseteq \overline{E} \setminus E',\  \lvert E'' \rvert = k  \right\} $ \\
\> $ E' \leftarrow E' \cup E_i$ \\
return ($E_1, \ldots, E_i, E \setminus \bigcup_{j=0}^i E_j$)
\end{tabbing}
\end{algorithm}

Incremental network design problems have been studied for the $s$-$t$ shortest path
problem~\cite{BaxterElgindyErnstKalinowskiSavelsbergh-2014-Incrementalnetworkdesign} and for the
maximum flow problem~\cite{KalinowskiMatsypuraSavelsbergh-2015-Incrementalnetworkdesign}. In both
cases, it was found that even the basic version of the incremental network design problem is
NP-complete. For the natural heuristics described above it has been shown that for the shortest path
problem, neither yields a constant factor approximation algorithm, but that for the maximum flow
problem with the additional restriction that all arcs have unit capacity, quickest-to-ultimate
yields a 2-approximation algorithm and quickest-improvement yields a $3/2$-approximation algorithm.

These results have raised two questions: (1) Does there exist a network optimization problem for
which the incremental design problem is polynomially solvable? (2) Does there exist a network
optimization problem for which either quickest-to-ultimate or quickest-improvement solves the
incremental design problem optimally?

In this paper, we answer both questions in the affirmative. We show that the basic version of the
incremental network design problem with minimum spanning trees is solved by both
quickest-improvement and quickest-to-ultimate. Note that in a slightly more general setting the
network design over time for minimum spanning trees is NP-complete. For instance, in the variant
considered in~\cite{NurreSharkey-2014-Integratednetworkdesign}, the construction of an edge can take
multiple time periods and the objective is to minimize a weighted sum of of the weights of minimum
spanning trees. This problem is proved to be NP-complete by a reduction from
\textsc{Partition}.

The \emph{incremental network design problem with minimum spanning trees} (IND-MST) is defined as
follows. For a given graph $G=(V, E)$, a weight function $w:E\to\reals$, and a set of existing edges
$E_0$, such that the subgraph $G=(V,E_0)$ is connected, find a sequence $X_0,\,X_1,\ldots,\,X_T$ of
spanning trees which minimizes the sum of the weights $w(X_0)+\cdots+w(X_T)$ subject to the
condition that $X_0\subseteq E_0$ and $\left\lvert X_i\cap(E\setminus(E_0\cup X_1\cup\cdots\cup
  X_{i-1}))\right\rvert\leqslant 1$ for $1\leqslant i\leqslant T$, i.e., at most one edge from
$E\setminus E_0$ might be added in each step. This has some similarity with the problem of
maintaining a dynamic minimum spanning tree while the network data changes~\cite{chin1978algorithms,
  eppstein1994offline, eppstein1997sparsification, frederickson1985data, spira1975finding}.  In
contrast to these dynamic minimum spanning tree problems, in our setting the network changes are not
given as input, but are part of the decisions to be made. We will show that IND-MST can be solved by
a greedy algorithm. This is a consequence of the corresponding result for the incremental matroid
design problem with minimum weight matroid bases, which is stated in Section~\ref{sec:min_basis} and
proved in Section~\ref{sec:proof}.

\section{Incremental matroid design}\label{sec:min_basis}

Let $M=(E,\mathcal I)$ be a matroid of rank $r$, where $E$ is the ground set, and $\mathcal
I\subseteq 2^E$ is the collection of independent sets. We follow the notation of
Schrijver~\cite{Schrijver-2003-Combinatorialoptimizationpolyhedra}: the rank of a matroid $M$ is denoted by rk($M$),
minimal dependent sets are called \emph{circuits}, for $A\subset E$ and $e\in E$ we write $A + e = A
\cup \{e\}$ and $A - e = A \setminus \{e\}$, and we denote the closure of a set $A\subseteq E$ by
$\cl(A)$:
\[\cl(A)=\{e\in E\ :\ \rk(A+e)=\rk(A)\}.\]
For a positive integer $n$, let $[n]$ denote the set $\{1,\ldots,n\}$. An important tool
in our arguments is the following \emph{strong exchange property} which was first proved by
Brualdi~\cite{brualdi1969comments}.
\begin{description}
\item[Strong exchange property.] If $X$ and $Y$ are bases of a matroid $M$ and $e\in X$, then there
  exists an element $e'\in Y$ such that $X-e+e'$ and $Y-e'+e$ are bases of $M$.
\end{description}
As additional input, we are given a weight function $w:E\to\reals$ and a subset $E_0\subset E$ such
that $E_0$ contains a basis of $M$. The weight function $w$ can be naturally extended to the power
set of $E$ by setting $w(X)=\sum_{e\in X}w(e)$ for all $X\subseteq E$. We define a function
$f:2^{E\setminus E_0}\to\reals$ by
\[f(A)=\min\{w(X)\ :\ X\subseteq E_0\cup A\text{ is a basis of }M\}\qquad\text{for }A\subseteq
E\setminus E_0.\] The \emph{incremental matroid design problem with minimum weight bases} (IMD-MWB)
problem for the time horizon $T=\lvert E\setminus E_0\rvert$ is the following optimization problem:
\begin{equation}\label{eq:imst}
\min\left\{\sum_{i=0}^T f(A_i)\ :\ A_0=\emptyset,\ \lvert A_i\setminus A_{i-1}\rvert=1\text{ for }i\in[T]\right\}.
\end{equation}
For a basis $X$, a pair $(e,e')\in X\times(E\setminus X)$ is called an \emph{exchange pair for $X$}
if $X-e+e'$ is another basis. It is called an \emph{optimal} exchange pair if $w(e)-w(e')$ is
maximum.  Algorithm~\ref{alg:greedy_1} is a natural greedy strategy for solving IMD-MWB, where the
output defines the sets $A_i$ in~(\ref{eq:imst}) via $A_i=\{e'_1,\ldots,e'_i\}$ for $i\leqslant k$
and $A_i=A_{i-1}+e'$ for arbitrary $e'\in E\setminus(E_0\cup A_{i-1})$ for $k+1\leqslant i\leqslant
T$. This corresponds to using quickest-improvement.
\begin{algorithm}[htb]
\caption{Greedy algorithm for the incremental minimum weight basis problem}\label{alg:greedy_1}
\begin{tabbing}
....\=....\=....\=....\=................... \kill \\
$k\leftarrow 0$\\
$w^*\leftarrow$ minimum weight of a basis of $M$\\
$X\leftarrow$ any minimum weight basis of the submatroid induced by $E_0$ (which is also a basis of $M$)\\
\textbf{while} $w(X) > w^*$ \textbf{do}\\
\> $k\leftarrow k+1$\\
\> $(e,e')\leftarrow$ an optimal exchange pair  for $X$\\
\> $X\leftarrow X-e+e'$\\
\> $e'_k\leftarrow e'$\\
return $(e'_1,\ldots,e'_k)$
\end{tabbing}
\end{algorithm}

Our main result is a consequence of the following theorem.
\begin{theorem}\label{thm:IMWB}
  Algorithm~\ref{alg:greedy_1} finds an optimal solution for the problem IMD-MWB.
\end{theorem}
If the second component $e'$ of an exchange pair $(e,e')$ for $X$ belongs to $Y\setminus X$, where
$Y \subseteq E$, then we call such a pair an \emph{exchange pair for $(X,Y)$}. Before proving
Theorem~\ref{thm:IMWB} we observe that the search for an optimal exchange pair can be restricted to
exchange pairs $(e,e')$ for $(X,Y)$, where $Y$ is a \textit{fixed} minimum weight basis of $M$.
This corresponds to using quickest-to-ultimate and leads to Algorithm~\ref{alg:greedy_2}.
\begin{algorithm}[htb]
\caption{Simplified greedy algorithm}\label{alg:greedy_2}
\begin{tabbing}
....\=....\=....\=....\=................... \kill \\
$X\leftarrow$ any minimum weight basis of the submatroid induced by $E_0$\\
$Y\leftarrow$ any minimum weight basis of the matroid $M$\\
\textbf{for} $k=1,\ldots,\lvert Y\setminus X\rvert$ \textbf{do}\\
\> $(e,e')\leftarrow$ an optimal exchange pair for $(X,Y)$\\
\> $X\leftarrow X-e+e'$\\
\> $e'_k\leftarrow e'$\\
return $(e'_1,\ldots,e'_k)$
\end{tabbing}
\end{algorithm}
\begin{corollary}\label{cor:IMWB_simplified}
Algorithm~\ref{alg:greedy_2} finds an optimal solution for the problem IMD-MWB.
\end{corollary}
\begin{proof}
  This follows from the claim that for any basis $X$ of $M$, and any minimum weight basis $Y$, there
  is an optimal exchange pair $(e,e')$ for $X$ with $e'\in Y$. Suppose the claim is false and let
  $(e,e')$ be an optimal exchange pair for $X$. By the strong exchange property applied to the bases
  $X'=X-e+e'$ and $Y$ and the element $e'\in X'$, there exists an $e''\in Y$ such that
  $X'-e'+e''=X-e+e''$ and $Y-e''+e'$ are bases. Our assumption implies $w(e'')>w(e')$, while from
  the minimality of $Y$ it follows that $w(e'')\leqslant w(e')$.
\end{proof}

\section{Proof of Theorem~\ref{thm:IMWB}}\label{sec:proof}
The following lemma is stated in~\cite{eppstein1994offline} for graphical matroids. The argument
works in general, and in order to make our presentation self-contained we include the short proof.
\begin{lemma}\label{lem:eppstein}
  Let $M=(E,\mathcal I)$ be a matroid, $E_0\subseteq E$, $A\subseteq E\setminus E_0$. In addition,
  let $M_0=(E_0,\mathcal I_0)$ and $M_A=(E_0\cup A,\mathcal I_A)$ be the matroids induced by $E_0$
  and $E_0\cup A$, respectively, i.e., $\mathcal I_0=\{X\cap E_0\ :\ X\in\mathcal I\}$ and $\mathcal
  I_A=\{X\cap (E_0\cup A)\ :\ X\in\mathcal I\}$, and let $X$ be a minimum weight basis for the
  matroid $M_0$. Then there exists a minimum weight basis $Y$ of $M_A$ such that $Y\subseteq X\cup A$.
\end{lemma}
Lemma~\ref{lem:eppstein} is proved by iterating the next lemma which states that a single element
exchange is sufficient in order to update the minimum weight basis after one potential element is
added.
\begin{lemma}\label{lem:basis_update}
  Let $A\subseteq E\setminus E_0$, and let $X_A$ be a minimum weight basis for the matroid $M_A$
  induced by $E_0\cup A$. Then, for every $e\in E\setminus(E_0\cup A)$, the set $X_A+e-e'$ is a minimum
  weight basis for the matroid $M_{A+e}$ induced by $E_0\cup A+e$, where $e'$ is an element of
  maximum weight in the circuit of $X_A+e$.
\end{lemma}
\begin{proof}
  Suppose the statement is false and let $Y$ be a basis of $M_{A+e}$ with
  $w(Y)<w(X_A)+w(e)-w(e')$. Then $e\in Y$, and by the strong exchange property, there exists
  $e''\in X_A$ such that $Y+e''-e$ and $X_A+e-e''$ are bases. The choice of $e'$ implies
  $w(e'')\leqslant w(e')$, while minimality of $X_A$ and our assumption on $Y$ imply that
\begin{multline*}
w(X_A)\leqslant w(Y-e+e'')=w(Y)-w(e)+w(e'')\\
<w(X_A)+w(e)-w(e')-w(e)+w(e'')=w(X_A)-w(e')+w(e''),
\end{multline*}

hence $w(e'')>w(e')$, and this contradiction concludes the proof.
\end{proof}
\begin{proof}[Proof of Lemma~\ref{lem:eppstein}] Let $A=\{e_1,\ldots,e_k\}$, set $A_0=\emptyset$ and $A_i=\{e_1,\ldots,e_i\}$ for $i=1,\ldots,k$, and apply Lemma~\ref{lem:basis_update} with $A=A_i$, $e=e_{i+1}$ for $i=0,\ldots,k-1$.
\end{proof}

For $t=0,1,\ldots,T$, let
\[F_t=\min\left\{f(A)\ :\ A\subseteq E\setminus E_0,\ \lvert A\rvert=t\right\}.\] Note that
$F_0>F_1>\cdots>F_t=F_{t+1}=\cdots=F_T$ for some $t\leqslant\min\{r,T\}$ and
Algorithm~\ref{alg:greedy_1} terminates with $k\geqslant t$. Clearly, $F_0+F_1+\cdots+F_T$ is a
lower bound for~(\ref{eq:imst}). The correctness of Algorithm~\ref{alg:greedy_1} follows from the
fact that it achieves this lower bound, which in turn is a consequence of the following extension
property.
\begin{lemma}\label{lem:nestedness}
  Let $t$ be the unique index with $F_{t-1}>F_t=F_{t+1}$, let $k<t$, and let
  $A,B\subseteq E\setminus E_0$ satisfying the following conditions:
  \begin{enumerate}
  \item $\lvert A\rvert= k$ and $\lvert B\rvert=k+1$,  
  \item $A$ and $B$ are optimal solutions for the minimization problems definining $F_k$ and
    $F_{k+1}$, respectively, i.e., $f(A)=F_k$ and $f(B)=F_{k+1}$.
  \end{enumerate}
 Then there exists $e\in B\setminus A$ such that $f(A+e)=F_{k+1}$.
\end{lemma}
\begin{proof}
  Let $M_0$, $M_A$ and $M_B$ denote the submatroids induced by $E_0$, $E_0\cup A$ and $E_0\cup B$,
  respectively.  We have $\rk(M_0)=\rk(M_A)=\rk(M_B)=r$ because $E_0$ contains a basis of $M$.  By
  the optimality of $A$ and $B$ and since $F_{k+1}<F_k$, the sets $A$ and $B$ are contained in every
  minimum weight basis of $M_A$ and $M_B$, respectively, and by Lemma~\ref{lem:eppstein}, there are
  minimum weight bases $X$, $X_A$ and $X_B$ for these submatroids with $X_A\setminus X=A$ and
  $X_B\setminus X=B$. We define a bipartite digraph $(\mathcal U\cup\mathcal V,\mathcal A)$ with
  parts
\begin{align*}
\mathcal U &= X_A\setminus\cl((X_B\cap X)\cup A), & \mathcal V= X_B\setminus\cl((X_B\cap X)\cup A).
\end{align*}
Note that $\rk((X_B\cap X)\cup A)\leqslant\lvert(X_B\cap X)\cup A\rvert=r-1$, hence $\mathcal
U,\mathcal V\neq\emptyset$. Also, $\mathcal U\subseteq X\subseteq E_0$ and $\mathcal V\subseteq
B\subseteq E\setminus E_0$, hence $\mathcal U\cap\mathcal V=\emptyset$. The arc set is defined by
\begin{multline*}
\mathcal A = \left\{(e,e')\in\mathcal U\times\mathcal V\ :\ (e,e')\text{ is an exchange pair for }X_A\right\}\\
\cup\left\{(e,e')\in\mathcal V\times\mathcal U\ :\ (e,e')\text{ is an exchange pair for }X_B\right\}.
\end{multline*}
Let $e'\in\mathcal V$. Then $e'\in E\setminus(E_0\cup A)$, hence $e'\not\in X_A$ and $X_A+e'$
contains a circuit $C$. We claim that $C-e'\not\subseteq\cl((X_B\cap X)\cup A)$, and this implies
that there exist $e\in C\cap\mathcal U$, and consequently $(e,e')\in\mathcal A$. For the sake of
contradiction, suppose the claim is false and $C-e'\subseteq\cl((X_B\cap X)\cup A)$. From the fact
that $C$ is a circuit, it follows that
\[e'\in\cl(C-e')\subseteq\cl((X_B\cap X)\cup A)\] which is a contradiction to $e'\in\mathcal
V$. Similarly, if $e'\in\mathcal U$, then $e'\in X_A\setminus A\subseteq E_0$ and $e'\not\in X_B\cap
X$, which implies $e'\not\in X_B$, hence there exists a circuit $C$ in $X_B+e'$. As before, the
assumption that $C-e'$ is contained in $\cl((X_B\cap X)\cup A)$ leads to the contradiction
$e'\in\cl((X_B\cap X)\cup A)$. By this argument, for every $e'\in\mathcal U$ there exists an
$e\in\mathcal V$ with $(e,e')\in\mathcal A$. We conclude that every node in the digraph $(\mathcal
U\cup\mathcal V,\mathcal A)$ has positive indegree, thus the digraph contains a directed cycle, and
this implies that there are $e',e''\in\mathcal U$ and $e\in\mathcal V$ such that $(e',e)\in\mathcal
A$, $(e,e'')\in\mathcal A$, and $w(e')\geqslant w(e'')$.  From this we derive
\begin{multline*}
f(A+e)\leqslant w(X_A+e-e') =  f(A)+w(e)-w(e')\leqslant F_k+w(e)-w(e'')\\
\leqslant f(B-e)+w(e)-w(e'')\leqslant w(X_B-e+e'')+w(e)-w(e'')=w(X_B)=F_{k+1}.
\end{multline*}
The converse inequality $f(A+e)\geqslant F_{k+1}$ is obvious and this concludes the proof.
\end{proof}

\section{Run-time analysis}\label{sec:runtime}
A rough upper bound for the run-time of a naive implementation of Algorithm~\ref{alg:greedy_2} for
the IND-MST problem on a graph with $n$ vertices can be obtained as follows: The optimal exchange
pair in each step of the for-loop can be found in time $O(n^2)$ by running through $O(n)$ candidates
for $e'$ and then by determining the best partner $e$ for $e'$ in linear time. Since $Y\setminus X$
has size $O(n)$ this gives in total a run-time of $O(n^3)$ which dominates the time needed to find
the minimum spanning trees $X$ and $Y$. In the following, we provide a more thorough estimate for
the run-time.

In order to bound the time complexities of the problems IMD-MWB and IND-MST, we argue that the
initial basis $X$, the ultimate basis $Y$ and the list of exchange pairs $\mathcal E$ can be
determined simultaneously. The idea is to consider the elements of $E$ in order of nondecreasing
weights and to construct and maintain three independent sets $X$, $Y$ and $Z$ using the following
update rules:
\begin{enumerate}
\item An element $e\in E$ is added to $X$ if and only if $e\in E_0$ and the addition of $e$ does not
  create a circuit in $X$. Hence $X$ is an initial minimum weight basis when the algorithm
  terminates.
\item An element $e\in E$ is added to $Y$ if and only if the addition of $e$ does not create a
  circuit in $Y$. Hence $Y$ is an ultimate minimum weight basis when the algorithm terminates.
\item An element $e\in E$ is added to $Z$ if and only if it is added to $X$ or $Y$. We will prove
  the following fact: If an edge added to $Y$ does not create a circuit in $Y$ then it does not
  create a circuit in $Z$. By this fact, an edge e added to $Z$ can create a circuit $C$ in $Z$ only
  if it has been added to $X$. This implies $e \in E_0$ and $C$ must contain an element of
  $E\setminus E_0$ (otherwise $e$ would have created a circuit in $X$). In this situation, a maximum
  weight element $e'$ of $C\setminus E_0$ is removed from $Z$ to preserve the independence of
  $Z$. The pair $(e,e')$ is added to the set $\mathcal E$ of exchange pairs.
\end{enumerate}
To finish up, the set $\mathcal E$ of exchange pairs $(e,e')$ is ordered such that $w(e)-w(e')$ is
nonincreasing, and ties are broken in favor of the pair that was added to $\mathcal E$ last. More
precisely, we index $\mathcal E=\{(e_1,e'_1),\ldots,(e_k,e'_k)\}$, such that for every
$i\in\{1,\ldots,k-1\}$, we have either $w(e_{i})-w(e'_{i})>w(e_{i+1})-w(e'_{i+1})$, or
$w(e_{i+1})-w(e'_{i+1})=w(e_{i})-w(e'_{i})$ and the pair $(e_{i},e'_{i})$ is added to $\mathcal E$
after $(e_{i+1},e'_{i+1})$. A formal description can be found in Algorithm~\ref{alg:efficient}. For
the remainder of the section, let $k$ be the size of the set $\mathcal E$ returned by the algorithm.
\begin{algorithm}[htb]
\caption{Efficient solution of the problem IMD-MWB}\label{alg:efficient}
\begin{tabbing}
....\=....\=....\=....\=................... \kill \\
$X\leftarrow\emptyset$;\quad $Y\leftarrow\emptyset$;\quad $Z\leftarrow\emptyset$ \qquad // initialize three independent sets \\
$\mathcal E\leftarrow\emptyset$\qquad// initialize set of exchange pairs \\
\textbf{for} $e \in E$ \textbf{do} \qquad // in nondecreasing order of weight\\
\> \textbf{if} $e\in E_0$ \textbf{then}\\
\> \> \textbf{if} $X+e\in\mathcal I$ \textbf{then} \\
\> \> \> $X\leftarrow X+e$;\quad $Z\leftarrow Z+e$\\
\> \> \> \textbf{if} $Z$ contains a circuit $C$ \textbf{then} \\
\> \> \> \> $e'=\argmax\{w(e'')\ :\ e''\in C\setminus E_0\}$\\
\> \> \> \> $Z\leftarrow Z-e'$;\quad $\mathcal E\leftarrow\mathcal E\cup\{(e,e')\}$\\
\> \textbf{if} $Y+e\in\mathcal I$ \textbf{then} \\
\> \> \> $Y\leftarrow Y+e$;\quad $Z\leftarrow Z+e$\\
sort $\mathcal E$\\
return $X$ and $\mathcal E=\left\{(e_1,e'_1),\ldots,(e_k,e'_k)\right\}$
\end{tabbing}
\end{algorithm}

In order to show the correctness of Algorithm~\ref{alg:efficient}, we introduce some additional
notation. Let $X_0=X$ and $X_i=X_{i-1}-e_i+e'_i$ for $(e_i, e'_i) \in \mathcal E$,
$i=1,2,\ldots,k$. Let $m=\lvert E\rvert$. We say that $e$ has \emph{position $p$}, denoted by $\pos(e)=p$, if
$e$ is the element that is handled in the $p$-th iteration of the for-loop, $1 \leqslant p \leqslant
m$. Note that $\pos(e)<\pos(e')$ implies $w(e)\leqslant w(e')$.

Let $X^p$, $Y^p, Z^p$ denote the sets $X$, $Y$ and $Z$ after the
$p$-th iteration of the for-loop has been completed, $0 \leqslant p
\leqslant m$. It is obvious, that if $p\leqslant l$, then
\[X^p \subseteq X^l \qquad\text{and}\qquad Y^p \subseteq Y^l.\]

For $i=1,\ldots,k$, we denote the circuit that leads to the deletion of $e'_i$ from $Z$ by $C_i$. In
other words, if $\pos(e_i)=p$ then $C_i$ is the unique circuit in $Z^{p-1}+e_i$ and we have
$Z^p=Z^{p-1}+e_i-e_i'$. If $e$ is any element of $C_i$, then $\pos(e)\leqslant \pos(e_i)$ and hence
$w(e)\leqslant w(e_i)$. By the choice of $e_i'$ we have
\begin{equation}\label{C_i}
w(e)\leqslant w(e_i') \text{ for all } e \in C_i \setminus E_0.
\end{equation}
Let
\[H^p =\{e_i: i \in [k] \text{ and }\pos(e_i) \leqslant p\},
\qquad\text{and}\qquad
H'^p =\{e_i': i \in [k] \text{ and }\pos(e_i) \leqslant p\}.
\]
That is, the sets $H^p$ and $H'^p$ contain the edges involved in
exchanges occurring either before or when the edge in position $p$
is examined. More precisely, we have
\[(X^p \times Y^p) \cap \mathcal{E} = \{ (e_i, e'_i) : e_i \in H^p\mbox{ and } e'_i \in H'^p \}.\]

\begin{lemma}
\label{lm4}
We have for $0 \leqslant p \leqslant m$
\begin{enumerate}
\item $X^p \subseteq Z^p \subseteq X^p \cup Y^p$,
\item $\cl(Y^p)=\cl(Z^p)$,
\item $Y^p \setminus X^p \subseteq E \setminus E_0$,
\item $X^p,Y^p,Z^p \in {\mathcal I}$,
\item $X^p \setminus Y^p=H^p$,
\item $Y^p \setminus Z^p=H'^p$.
\end{enumerate}
\end{lemma}
\begin{proof}
  We proceed by induction on $p$. The case $p=0$ is trivial.  Now consider the step $p-1 \rightarrow
  p$. From $X^{p-1} \subseteq Z^{p-1} \subseteq X^{p-1} \cup Y^{p-1}$ we obtain directly $X^p
  \subseteq Z^p \subseteq X^p \cup Y^p$ since an element is added to $Z$ iff it is added to $X$ or
  to $Y$ and an element is deleted from $Z$ only if this element does not belong to $E_0$, i.e., not
  to $X$.  Now we prove the other assertions. Let $e$ be the element with $\pos(e)=p$.

  If $X^p=X^{p-1}$ and $Y^p=Y^{p-1}$ then also $Z^p=Z^{p-1}$, $H^p=H^{p-1}$ and $H'^p=H'^{p-1}$ and
  the assertion follows from the induction hypothesis. So there are three main cases:

\begin{description}
\item[Case 1.] $X^p=X^{p-1}+e$ and $Y^p=Y^{p-1}$. Then $e \in \cl(Y^{p-1})$. By the induction
  hypothesis, $e \in \cl(Z^{p-1})$ and consequently $Z^p=(Z^{p-1}+e)-e'$ for some $e'$ in the unique
  circuit of $Z^{p-1}+e$ where $e' \subseteq E \setminus E_0$ and $e' \neq e$. Obviously, $Z^p \in
  {\mathcal I}$ and $\cl (Z^p)=\cl(Z^{p-1}+e)=\cl(Z^{p-1})=\cl(Y^{p-1})=\cl(Y^p)$. Obviously, $X^p$,
  $Y^p \in {\mathcal I}$ and $Y^{p}\setminus X^p \subseteq Y^{p-1} \setminus X^{p-1} \subseteq E
  \setminus E_0$. Finally, $X^p \setminus Y^p=(X^{p-1}+e)\setminus Y^{p-1}=(X^{p-1}\setminus
  Y^{p-1})+e=H^{p-1}+e=H^p$ and $Y^p \setminus Z^p=Y^{p-1}\setminus
  ((Z^{p-1}+e)-e')=(Y^{p-1}\setminus Z^{p-1})+e'=H'^{p-1}+e'=H'^p$.
\item[Case 2.] $X^p=X^{p-1}$ and  $Y^p=Y^{p-1}+e$. Then
$Z^p=Z^{p-1}+e$ and hence $\cl(Y^p)=\cl(Z^p)$. Obviously, $X^p$,
$Y^p \in {\mathcal I}$ and $e \notin \cl(Y^{p-1})$, i.e., $e \notin
\cl(Z^{p-1})$. Thus $Z^p \in {\mathcal I}$.

If $e \in E_0$ then $e \in \cl(X^{p-1}) \subseteq \cl(Z^{p-1})=\cl(Y^{p-1})$, which contradicts
$Y^{p-1}+e \in \mathcal{I}$. Hence $e \notin E_0$. Then $Y^{p}\setminus X^p = (Y^{p-1}+e) \setminus
X^{p-1} \subseteq E \setminus E_0$. Finally, $X^p \setminus Y^p=X^{p-1}\setminus
(Y^{p-1}+e)=X^{p-1}\setminus Y^{p-1}=H^{p-1}=H^p$ and $Y^p \setminus Z^p=Y^{p-1}\setminus
Z^{p-1}=H'^{p-1}=H'^p$.

\item[Case 3.] $X^p=X^{p-1}+e$ and $Y^p=Y^{p-1}+e$. Then $X^p$, $Y^p \in {\mathcal I}$ and $e \notin
  \cl(Y^{p-1})$, i.e., $e \notin \cl(Z^{p-1})$. Thus $Z^{p-1}+e \in {\mathcal I}$. This implies
  $Z^p=Z^{p-1}+e$ and consequently $\cl(Y^p)=\cl(Z^p)$ as well as $Y^{p}\setminus X^p = Y^{p-1}
  \setminus X^{p-1} \subseteq E \setminus E_0$. Finally, $X^p \setminus Y^p=X^{p-1} \setminus
  Y^{p-1}=H^{p-1}=H^p$ and $Y^p \setminus Z^p=Y^{p-1} \setminus Z^{p-1}=H'^{p-1}=H'^p$. \qedhere
\end{description}
\end{proof}

In the following, we denote the three independent sets that the algorithm terminates with by $X, Y$,
and $Z$, i.e., $X = X^m$, $Y = Y^m$ and $Z = Z^m$.

\begin{corollary}
\label{cor2}
We have
\begin{enumerate}
\item $X=Z$,
\item $X\setminus Y=\{e_1,\dots,e_k\}$,
\item $Y\setminus X=\{e_1',\dots,e_k'\}$.
\end{enumerate}
\end{corollary}
\proof By our general supposition that $E_0$ contains a basis, the set $X$ is a basis for
$M$. Since, by Lemma~\ref{lm4}, $Z$ is independent and $X \subseteq Z$, we have $X=Z$. Again, by
Lemma~\ref{lm4}, $X \setminus Y=X^m \setminus Y^m=H^m=\{e_1,\dots,e_k\}$ and $Y \setminus X=Y
\setminus Z = Y^m \setminus Z^m=H'^m=\{e_1',\dots,e_k'\}$. \qed

\begin{corollary}\label{cor3} 
  If $\pos(e_i)=p$, $i \in [k]$, then $C_i \cap H'^p=\{e_i'\}$.
\end{corollary}
\begin{proof}
  We have $C_i \subseteq Z^{p-1}+e_i=Z^p+e_i'$ and thus $C_i \cap H'^p=C_i \cap (Y^p \setminus Z^p)
  \subseteq (Z^p+e_i') \cap (Y^p \setminus Z^p)=\{e_i'\}$. Clearly, $e_i' \in C_i \cap H'^p$.
\end{proof}

\begin{lemma}\label{lm5} 
  Let $C$ be a circuit in $X \cup Y$. Let
  \[e^*=\argmax\{\pos(e): e \in C\}.\] 
Then $e^* = e_l$ for some $l \in [k]$.
\end{lemma}
\begin{proof}
  Let $\pos(e^*)=q.$ Then $C-e^* \subseteq X^{q-1} \cup Y^{q-1}$. This implies $e^* \in \cl(X^{q-1}
  \cup Y^{q-1})$ and from Lemma~\ref{lm4} we obtain $e^* \in \cl(X^{q-1} \cup
  Z^{q-1})=\cl(Z^{q-1})$. Since $e^* \in X^q \cup Y^q$, but $e^* \notin X^{q-1} \cup Y^{q-1}$ we
  must have $X^q=X^{q-1}+e^*$ and $Z^q=Z^{q-1}+e^*-e'$ for some $e' \in E \setminus E_0$. Hence
  $e^*=e_l$ for some $l \in [k]$.
\end{proof}

\begin{lemma}\label{lm6} 
  We have $C_i \subseteq X \cup \{e_1',\dots,e_i'\}$ for every $i \in [k]$.
\end{lemma}
\begin{proof}
  Using Corollary~\ref{cor2}, we obtain $C_i \subseteq X \cup Y=X \cup (Y \setminus X)= X \cup
  \{e_1',\dots,e_k'\}$. Assume that there is some $j>i$ such that $e_j' \in C_i$. Then
  $\pos(e_j)>\pos(e_i)$ and thus $w(e_j) \geqslant w(e_i)$. From~(\ref{C_i}) and $e_j'$ in
  $C_i\setminus E_0$, we obtain $w(e_i') \geqslant w(e_j')$. Consequently, $w(e_j)-w(e_j') \geqslant
  w(e_i)-w(e_i')$, a contradiction to the ordering of ${\mathcal E}$.
\end{proof}

\begin{lemma}
  For $i \in [k]$, the pair $(e_i,e_i')$ is an exchange pair for $(X_{i-1},Y)$.
\end{lemma}
\begin{proof}
  Clearly, $e_i' \in Y\setminus (X \cup \{e_1', \dots, e_{i-1}'\}) \subseteq Y \setminus
  X_{i-1}$. We have to show that $e_i$ lies in the unique circuit of $X_{i-1}+e_i'=(X \setminus
  \{e_1,\dots,e_{i-1}\}) \cup \{e_1',\dots,e_i'\}$. An equivalent statement is that there is a
  circuit in $(X \setminus \{e_1,\dots,e_{i-1}\}) \cup \{e_1',\dots,e_i'\}$ containing $e_i$.  From
  Lemma~\ref{lm6}, we know that there is at least a circuit in $X \cup \{e_1',\dots,e_i'\}$
  containing $e_i$, namely $C_i$. For a circuit $C$, let
\[\mu_C^i= \max\{\pos(e_j): j \in [i-1] \text{ and } e_j \in C\},\]
where the maximum extended over an empty set is defined to be $-\infty$.

Assume that there is no circuit in $(X \setminus\{e_1,\dots,e_{i-1}\}) \cup \{e_1',\dots,e_i'\}$
containing $e_i$. Then we choose a circuit $C$ in $X \cup \{e_1',\dots,e_i'\}$ containing $e_i$
such that $\mu_C^i$ is minimal. By our assumption, $\mu_C^i$ is finite and there exists an integer
$l \in [i-1]$ such that $e_l \in C$ and $\pos(e_l)=\mu_C^i$. The circuit $C_l$ also contains $e_l$
and is contained in $X \cup \{e_1',\dots,e_i'\}$ in view of Lemma~\ref{lm6} and $l < i$. Note that
$\pos(e)<\pos(e_l)$ for all $e \in C_l-e_l$. Now $C$ and $C_l$ are two distinct circuits
with $e_l\in C\cap C_l$, and therefore there is also a circuit $\widetilde{C} \subseteq C
\cup C_l - e_l$. Obviously, $\widetilde{C} \subseteq X \cup \{e_1',\dots,e_i'\}$ and
$\mu_{\widetilde{C}}^i < \mu_C^i$, a contradiction to the choice of $C$.
\end{proof}

\begin{lemma}
For $i \in [k]$ the pair $(e_i,e_i')$ is an optimal exchange pair for $(X_{i-1},Y)$.
\end{lemma}
\begin{proof}
  Assume that $(e_i,e_i')$ is not optimal. Then there is a better exchange pair
  $(\hat{e},\,\hat{e}')$ for $(X_{i-1},Y)$. We have $\hat{e}' \in Y \setminus X_{i-1}=(Y \setminus
  (X \cup \{e_1',\dots,e_{i-1}'\})) \cup (Y \cap \{e_1,\dots,e_{i-1}\})$. From Corollary~\ref{cor2}
  it follows that $\hat{e}'=e_j'$ for some $j \in [k]$. Since $e_1',\dots,e_{i-1}'\notin Y \setminus
  X_{i-1}$,
\begin{equation}\label{ji}
j \geqslant i.
\end{equation}

Let $\widehat{C}$ be the unique circuit in $X_{i-1}+e_j'$ (containing $\hat{e}$) and let
$e^*=\argmax\{\pos(e): e \in \widehat{C}\}$. Then $w(e^*) \geqslant w(\hat{e})$. By Lemma~\ref{lm5},
$e^*=e_l$ for some $l \in [k]$. In particular, $e^* \neq e_j'$ and thus $(e^*,e_j')$ is also an
exchange pair for $(X_{i-1},Y)$. Since $e_1,\dots,e_{i-1} \notin X_{i-1}$,
\begin{equation}
\label{li}
l \geqslant i.
\end{equation}
We have (with $e^*=e_l$) $w(e_l)-w(e_j') \geqslant w(\hat{e})-w(e_j')$ and hence
\begin{equation}
\label{lj}
w(e_l)-w(e_j') > w(e_i)-w(e_i').
\end{equation}
By the ordering of $\mathcal E$ and by~(\ref{li}),
\begin{equation}
\label{il}
w(e_i)-w(e_i') \geqslant  w(e_l)-w(e_l').
\end{equation}
The inequalities~(\ref{lj}) and~(\ref{il}) imply
\begin{equation}\label{jl}
w(e_j') < w(e_l').
\end{equation}

Let $p=\pos(e_l)$. Then $\pos(e_j) < p$, because otherwise $w(e_j) \geqslant w(e_l)$ and hence, in
view of~(\ref{ji}) and the ordering of $\mathcal{E}$,
\[w(e_l)-w(e_j')\leqslant w(e_j)-w(e_j') \leqslant w(e_i)-w(e_i'),\] 
a contradiction to~(\ref{lj}). Moreover, $\widehat{C} \subseteq X^p \cup Y^p$ since $e_l$ has
maximal position in $\widehat{C}$. For a circuit $C$ let
\begin{align*}
\alpha_C^p&=\max\{w(e_h'): h \in [k], e_h' \in C  \text{ and } \pos(e_h) \leqslant p\},\\
\nu_C &= \min\{\pos(e_h): h \in [k] \text{ and } e_h' \in C\},
\end{align*}
where the maximum (resp. minimum) extended over an empty set is defined to be $-\infty$
(resp. $\infty$). For $\widehat{C}$ these values are finite since $e_j' \in \widehat{C}$ and
$\pos(e_j) < p$.  Hence there is an integer $g \in [k]$ such that $e_g' \in \widehat{C}, \pos(e_g)
\leqslant p$ and $w(e_g')=\alpha_{\widehat{C}}^p$\,. Note that
\begin{equation}\label{gi}
g \leqslant i-1  \text{ or } g=j
\end{equation}
since $\widehat{C} \subseteq X \cup \{e_1',\dots,e_{i-1}',e_j'\}$.  We say that a circuit $C$ is
\emph{$p$-majorized} by a number $\alpha$ if
\[w(e_h') \leqslant \alpha \quad \text{for all }h \in [k] \text{ with }e_h' \in C \text{ and }
\pos(e_h) \leqslant p.\] 
Note that $\widehat{C}$ is $p$-majorized by $\alpha_{\widehat{C}}^p$. Now
choose a circuit $C^*$ in $X^p \cup Y^p$ that contains $e_l$, is $p$-majorized by
$\alpha_{\widehat{C}}^p$ and has maximal $\nu$-value.

Note that, in view of Lemma~\ref{lm4} (parts 1 and 6), $C^* \subseteq X^p \cup Y^p \subseteq Z^p \cup Y^p = Z^p \cup
(Y^p \setminus Z^p)= Z^p \cup H'^p$. We have $\nu_{C^*} \leqslant p$ because otherwise $C^* \cap
H'^p= \emptyset$ and $C^*$ would be a circuit in $Z^p$, contradicting Lemma~\ref{lm4} (part 4).

Assume that $\nu_{C^*} < p$. Then choose $q \in [k]$ such that $e_q' \in C^*$ and
$\pos(e_q)=\nu_{C^*}$ and consider the circuit $C_q$.  We have $C_q \subseteq X^p \cup Y^p$ because
of $\pos(e_q) < p$.  Moreover, $w(e_h') \leqslant w(e_q')$ and $\pos(e_h) \geqslant \pos(e_q)$ for
all $h \in [k]$ with $e_h' \in C_q$ by the choice of $e_q'$ in $C_q$ and Corollary~\ref{cor3}. Since
$C^*$ is $p$-majorized by $\alpha_{\widehat{C}}^p$, in particular $w(e_q') \leqslant
\alpha_{\widehat{C}}^p$ and thus also $C_q$ is $p$-majorized by $\alpha_{\widehat{C}}^p$. By
definition, we have $C_q\subseteq Z^{p'-1}+e_q$ for $p'=\pos(e_q)<p=\pos(e_l)$, and therefore
$e_l\not\in C_q$. From $e'_q\in C^*\cap C_q$ and $e_l\in C^*\setminus C_q$ it follows that there is a circuit
$\widetilde{C} \subseteq C^* \cup C_q - e_q'$ containing $e_l$. Clearly, $\widetilde{C} \subseteq
X^p \cup Y^p$ and $\widetilde{C}$ is $p$-majorized by $\alpha_{\widehat{C}}^p$.  Obviously,
$\min\{\pos(e_h): h \in [k] \text{ and } e_h' \in C_q \cup C^*\}=\pos(e_q)$. Thus
$\nu_{\widetilde{C}}> \nu_{C^*}$, a contradiction to the choice of $C^*$.

Consequently, $\nu_{C^*}=p$. Since $e_l$ is the (unique) element of position $p$, necessarily $e_l'
\in C^*$. Because $C^*$ is $p$-majorized by $\alpha_{\widehat{C}}^p$, in particular $w(e_l')
\leqslant \alpha_{\widehat{C}}^p=w(e_g')$. From~(\ref{jl}) we obtain $g\neq j$, hence by~(\ref{gi})
and~(\ref{li}),
\begin{equation}\label{eq:g_less_l}
g \leqslant i-1 < l.
\end{equation}
On the other hand, the relation $\pos(e_g) \leqslant \pos(e_l)$ implies $w(e_g) \leqslant
w(e_l)$. Consequently,
\[w(e_l)-w(e_l')\geqslant w(e_g)-w(e_g'),\] 
and then the ordering of $\mathcal E$ implies $l\leqslant g$, which
contradicts~(\ref{eq:g_less_l}). Thus our first assumption was false, and the
proof is complete.
\end{proof}

For the runtime analysis, we assume that Algorithm~\ref{alg:efficient} gets as input the elements of
$E$ in nondecreasing order of weights, and that we have an algorithm which, for a given set
$X\subseteq E$, decides if $X\in\mathcal I$. Let $A$ denote an upper bound for the runtime of this
independence test. In each iteration of the for-loop there are three independence tests, and at most
one search for the element $e'\in Z$ to be removed. Since the final $X$ has $r$ elements, this
latter step is necessary at most $r$ times and can be done in time $O(A\lvert Z\rvert)=O(Ar)$, where
$r=\rk(M)$ is the rank of the matroid: just call the independence test for each of the sets $Z-e'$
in decreasing order of $w(e')$ until you find an independent set. So the for-loop terminates in time
$O(A(\lvert E\rvert+r^2))$. The final sorting of $\mathcal E$ takes time $O(r\log r)$. To summarize,
we have proved the following bound for the time complexity of the problem IMD-MWB.
\begin{theorem}\label{thm:complexity_IMD}
  For a matroid $M=(E,\mathcal I)$ of rank $r$, where the elements of $E$ are given in nondecreasing
  order of weights, the problem IMD-MWB can be solved in time $O(A(\lvert E\rvert+r^2))$, where $A$
  is a runtime bound for an algorithm that decides the independence of a set $X\subseteq E$.
\end{theorem}
Finally, we consider the special case where $M$ is a graphical matroid to solve our original problem
IND-MST.
\begin{theorem}\label{thm:complexity_IND}
  The problem IND-MST for a graph with $n$ vertices and $m$ edges can be solved in time
\[O(m+n\log n).\]
\end{theorem}
\begin{proof}
  Using Fibonacci heaps~\cite{Fredman.Tarjan_1987_Fibonacciheapsand}, minimum spanning trees $X$ and $Y$ for the graphs $G=(V,E)$ and
  $G_0=(V,E_0)$, respectively, can be constructed in time $O(m+n\log
  n)$. We can then run Algorithm~\ref{alg:efficient} on
  the graph $G=(V,X\cup Y)$ which has only $O(n)$ edges. Using dynamic
  trees~\cite{Sleator.EndreTarjan_1983_datastructuredynamic} to represent $Z$, each of the exchange
  pairs $(e,e')$ can be found in time $O(\log n)$: if $e=\{u,v\}$ is the edge that creates a circuit in
  $Z$ then $e'$ is an edge of maximum weight on the path between $u$ and $v$ in $Z$. 
\end{proof}

\medskip

\textbf{Acknowledgment.} We thank the anonymous referees for valuable comments that helped improve
and clarify the presentation of the results.


\begin{thebibliography}{10}

\bibitem{averbakh2012emergency}
I.~Averbakh.
\newblock Emergency path restoration problems.
\newblock {\em Discrete Optimization}, 9(1):58--64, 2012.

\bibitem{averbakh2012flowtime}
I.~Averbakh and J.~Pereira.
\newblock The flowtime network construction problem.
\newblock {\em IIE Transactions}, 44(8):681--694, 2012.

\bibitem{AverbakhPereira-2015-Networkconstructionproblems}
I.~Averbakh and J.~Pereira.
\newblock Network construction problems with due dates.
\newblock {\em European Journal of Operational Research}, 244:715--729, 2015.

\bibitem{BaxterElgindyErnstKalinowskiSavelsbergh-2014-Incrementalnetworkdesign}
M.~Baxter, T.~Elgindy, A.~T. Ernst, T.~Kalinowski, and M.~W. Savelsbergh.
\newblock Incremental network design with shortest paths.
\newblock {\em European Journal of Operational Research}, 238(3):675--684,
  2014.

\bibitem{brualdi1969comments}
R.~A. Brualdi.
\newblock Comments on bases in dependence structures.
\newblock {\em Bulletin of the Australian Mathematical Society}, 1(2):161--167,
  1969.

\bibitem{cavdaroglu2013integrating}
B.~Cavdaroglu, E.~Hammel, J.~Mitchell, T.~Sharkey, and W.~Wallace.
\newblock Integrating restoration and scheduling decisions for disrupted
  interdependent infrastructure systems.
\newblock {\em Annals of Operations Research}, 203(1):279--294, 2013.

\bibitem{chin1978algorithms}
F.~Chin and D.~Houck.
\newblock Algorithms for updating minimal spanning trees.
\newblock {\em Journal of Computer and System Sciences}, 16(3):333--344, 1978.

\bibitem{eppstein1994offline}
D.~Eppstein.
\newblock Offline algorithms for dynamic minimum spanning tree problems.
\newblock {\em Journal of Algorithms}, 17(2):237--250, 1994.

\bibitem{eppstein1997sparsification}
D.~Eppstein, Z.~Galil, G.~F. Italiano, and A.~Nissenzweig.
\newblock Sparsification -- a technique for speeding up dynamic graph
  algorithms.
\newblock {\em Journal of the ACM}, 44(5):669--696, 1997.

\bibitem{frederickson1985data}
G.~N. Frederickson.
\newblock Data structures for on-line updating of minimum spanning trees, with
  applications.
\newblock {\em SIAM Journal on Computing}, 14(4):781--798, 1985.

\bibitem{Fredman.Tarjan_1987_Fibonacciheapsand}
M.~L. Fredman and R.~E. Tarjan.
\newblock Fibonacci heaps and their uses in improved network optimization
  algorithms.
\newblock {\em Journal of the ACM}, 34(3):596--615, 1987.

\bibitem{KalinowskiMatsypuraSavelsbergh-2015-Incrementalnetworkdesign}
T.~Kalinowski, D.~Matsypura, and M.~W. Savelsbergh.
\newblock Incremental network design with maximum flows.
\newblock {\em European Journal of Operational Research}, 242(1):51--62, 2015.

\bibitem{kerivin2005design}
H.~Kerivin and A.~R. Mahjoub.
\newblock Design of survivable networks: A survey.
\newblock {\em Networks}, 46(1):1--21, 2005.

\bibitem{magnanti1984network}
T.~L. Magnanti and R.~T. Wong.
\newblock Network design and transportation planning: Models and algorithms.
\newblock {\em Transportation Science}, 18(1):1--55, 1984.

\bibitem{NurreCavdarogluMitchellSharkeyWallace-2012-Restoringinfrastructuresystems}
S.~G. Nurre, B.~Cavdaroglu, J.~E. Mitchell, T.~C. Sharkey, and W.~A. Wallace.
\newblock Restoring infrastructure systems: An integrated network design and
  scheduling ({INDS}) problem.
\newblock {\em European Journal of Operational Research}, 223(3):794--806,
  2012.

\bibitem{NurreSharkey-2014-Integratednetworkdesign}
S.~G. Nurre and T.~C. Sharkey.
\newblock Integrated network design and scheduling problems with parallel
  identical machines: Complexity results and dispatching rules.
\newblock {\em Networks}, 63(4):306--326, 2014.

\bibitem{Schrijver-2003-Combinatorialoptimizationpolyhedra}
A.~Schrijver.
\newblock {\em Combinatorial optimization: polyhedra and efficiency}, volume~24
  of {\em Algorithms and Combinatorics}.
\newblock Springer, 2003.

\bibitem{SharkeyCavdarogluNguyenHolmanMitchellWallace-2015-Interdependentnetworkrestoration}
T.~C. Sharkey, B.~Cavdaroglu, H.~Nguyen, J.~Holman, J.~E. Mitchell, and W.~A.
  Wallace.
\newblock Interdependent network restoration: On the value of
  information-sharing.
\newblock {\em European Journal of Operational Research}, 244(1):309--321,
  2015.

\bibitem{Sleator.EndreTarjan_1983_datastructuredynamic}
D.~D. Sleator and E.~R. Tarjan.
\newblock A data structure for dynamic trees.
\newblock {\em Journal of Computer and System Sciences}, 26(3):362--391, Jun
  1983.

\bibitem{spira1975finding}
P.~M. Spira and A.~Pan.
\newblock On finding and updating spanning trees and shortest paths.
\newblock {\em SIAM Journal on Computing}, 4(3):375--380, 1975.

\end{thebibliography}

\end{document}